\documentclass{amsart}
\usepackage[T1]{fontenc}
\usepackage[utf8]{inputenc}
\newcommand{\Bl}{\operatorname{Bl}}
\usepackage{multicol}
\usepackage{nicematrix}

\usepackage[english]{babel}
\usepackage[notcite,notref]{showkeys}

\usepackage[title]{appendix}
\usepackage{hyperref}  
                 \hypersetup{ pdfborder={0 0 0}, 
                              colorlinks=true, 
                              linktoc=page,
                              pdfauthor={H.Awada, M. Bolognesi, R.Laterveer and C.Pedrini},
                              pdftitle={K3 surfaces and cubic fourfolds with Abelian motive}
                            }

\usepackage{amssymb}
\usepackage{amsthm}
\usepackage{graphicx}
\usepackage{import}
\usepackage{amsmath}
\usepackage{amscd}
\usepackage{mathtools}
\usepackage{verbatim}
\usepackage[all]{xy}
\usepackage{tikz-cd}
\usepackage{color}
\usepackage{geometry,mathtools}
\newcommand{\sF}{\mathcal{F}}
\newcommand{\sG}{\mathcal{G}}
\newcommand{\sM}{\mathcal{M}}
\newcommand{\sN}{\mathcal{N}}
\newcommand{\sC}{\mathcal{C}}

\newcommand{\Sing}{\operatorname{Sing}}

\newtheorem{thm}{Theorem}[section]
\newtheorem*{thm*}{Theorem}
\newtheorem{prop}[thm]{Proposition}
\newtheorem*{prop*}{Proposition}
\newtheorem{conj}[thm]{Conjecture}
\newtheorem{lem}[thm]{Lemma}
\newtheorem{cor}[thm]{Corollary}
\newtheorem{defi}[thm]{Definition}
\newtheorem{rem}[thm]{Remark}

\newcommand{\Q}{\mathbf{Q}}

\newcommand{\Jac}{{\operatorname{Jac}}}

\newcommand{\Pic}{\operatorname{Pic}}

\newcommand{\Sym}{\operatorname{Sym}}

\newcommand{\un}{\mathbf{1}}
\renewcommand{\P}{\mathbb{P}}
\newcommand{\Z}{\mathbb{Z}}
\newcommand{\M}{\mathcal{M}}

\newcommand{\C}{\mathcal{C}}

\begin{document}

\title{K3 surfaces and cubic fourfolds with Abelian motive}

\author[H.Awada]{Hanine Awada}
\address{Institut Montpellierain Alexander Grothendieck \\ %
CNRS \\ %
Universit\'e de Montpellier \\ %
Case Courrier 051 - Place Eug\`ene Bataillon \\%
34095 Montpellier Cedex 5 f\\ %
France}
\email{hanine.awada@umontpellier.fr}

\author[M. Bolognesi]{Michele Bolognesi}
\address{Institut Montpellierain Alexander Grothendieck \\ %
CNRS \\ %
Universit\'e de Montpellier \\ %
Case Courrier 051 - Place Eug\`ene Bataillon \\ %
34095 Montpellier Cedex 5 \\ %
France}
\email{michele.bolognesi@umontpellier.fr}

\author[Robert Laterveer]
{Robert Laterveer}
\address{Institut de Recherche Math\'ematique Avanc\'ee,
CNRS -- Universit\'e 
de Strasbourg,\
7 Rue Ren\'e Des\-car\-tes, 67084 Strasbourg CEDEX,
FRANCE.}
\email{robert.laterveer@math.unistra.fr}

\author[C. Pedrini]{Claudio Pedrini}
\address{Dipartimento di Matematica \\ %
Universit\`a degli Studi di Genova \\ %
Via Dodecaneso 35 \\ %
16146 Genova \\ %
Italy}
\email{pedrini@dima.unige.it}

\begin{abstract}
We show that cubic fourfolds with lattice of algebraic 2-cycles of rank greater than 19 have abelian and finite dimensional (in the sense of Kimura) Chow motive.
This also implies Abelianity and finite dimensionality of the motive of related hyperK\"ahler varieties, such as the Fano variety of lines and the LLSvS 8fold. A similar remark allows us to show the Abelianity of the motive of an infinity of LSV 10folds, and of other hyperKähler 10folds associated to the twisted intermediate Jacobian fibration of cubic fourfolds with an associated K3 surface.
After that, starting from certain 4-dimensional families of K3 surfaces, we construct two families of Fano varieties whose Chow motive is finite dimensional. Varieties from the first family are some quadric surface fibrations, and contain the finite dimensional transcendental motive of a K3 surface. Varieties from the second family are singular cubic fourfolds, and their motives are Schur-finite and Abelian in Voevodsky's triangulated category of motives.

\end{abstract}

\maketitle

\tableofcontents

{}\section{Introduction} \label{sec:introduction}

Let $\M_{rat}(\mathbf{C})$ be the category of Chow motives with $\mathbf{Q}$-coefficients and let $\M^{Ab}_{rat}(\mathbf{C})$ be the strictly full, thick, rigid, tensor subcategory of $\M_{rat}(\mathbf{C})$ generated by the motives of Abelian varieties.  All the examples of motives that have been proven to be finite-dimensional, in the sense of Kimura-O'Sullivan, belong to the category $\M^{Ab}_{rat}(\mathbf{C})$. More precisely, the following classes of smooth projective varieties are known to have motives belonging to $\M^{Ab}_{rat}(\mathbf{C})$:

\begin{enumerate}

\item projective spaces, Grassmannian varieties, projective homogeneous
varieties, toric varieties;
\item smooth projective curves;
\item Kummer K3 surfaces;
\item K3 surfaces with Picard numbers at least 19;
\item K3 surfaces  with a non-symplectic group of automorphisms  acting trivially on the algebraic cycles: K3 surfaces satisfying these conditions have 
Picard numbers equal to $ 2, 4, 6, 10, 12, 16,$ $ 18, 20$, see \cite{Ped};
\item Hilbert schemes of points on  Abelian surfaces;
\item Fermat hypersurfaces ;
\item Cubic 3-folds and their Fano surfaces of lines, see \cite{GG} and  \cite{Diaz}.
\end{enumerate}

In this paper, we consider the case of cubic fourfolds, in relation with certain families of K3 surfaces.

\smallskip

According to  Kimura's Conjecture on finite dimensional motives and by the results in \cite{An} the Chow motive $h(S)$  of a complex K3 surface  $S$ should be of Abelian type. Also, by the work of Kuga and Satake, if one assumes the Hodge conjecture, every K3 surface over $\mathbf{C}$ is of Abelian Hodge type, i.e. there exists an algebraic correspondence between any K3 surface $S$ and an associated Abelian variety,  the {\it Kuga-Satake variety} $K(S)$.  This would imply that the motive of any such K3 surface is Abelian, i.e. it lies in the subcategory $\sM^{Ab}_{rat}(\mathbf{C})$ of the (covariant) category $\sM_{rat}(\mathbf{C})$ of Chow motives generated by the motives of curves. On the other hand, cubic fourfolds have been a very active field of research in the last few years, for several reasons. The rationality of the generic cubic is a very classical, and still unanswered, question in algebraic geometry. A lot of energy has been spent in order to find a good invariant that would detect the required birational properties, and tentatives have been made via Hodge theory, derived categories, Chow motives, algebraic cycles, etc. In most of these papers K3 surfaces appear as an important presence (in the cohomology, derived category, related hyperK\"ahler varieties, ...) whenever the cubic fourfolds are rational, or suspected to be rational.

\smallskip
 
Hassett \cite{Ha1} introduced the notion of "\textit{special}" cubic fourfold, that is a cubic that contains an algebraic surface not homologous to a complete intersection. These fourfolds form a countable infinite union of divisors $\C_d$ called \it Hassett's Noether-Lefschetz divisors \rm (for short Hassett divisors) inside the moduli space $\C$ of smooth cubic fourfolds, which  is a 20-dimensional quasi-projective variety. Hassett showed that $\C_d$ is irreducible and nonempty if and only if $d\geq8$ and $d \equiv 0,2\ [6]$. Only few $\C_d$ have been defined explicitly in terms of surfaces contained in a general element of these divisors (see \cite{Ha1}, \cite{Ha2}, \cite{MR3934590}, \cite{russo2019trisecant} , \cite{dp6}, \cite{MR3968870}). More recently, the first and second named authors have deployed similar techniques to study the birational geometry of universal cubics over certain divisors $\C_d$ \cite{AB20}. Hassett \cite{Ha1} proved that, for an infinite set of values of $d$, one can associate a polarized K3 surface $(S,f)$ of degree $d$ to a cubic fourfold in $\C_d$. This is true for $d$ satisfying $4 \nmid  d, 9 \nmid d$, and $ p \nmid d$ for any odd prime number $p \equiv 2\ [3]$, and the association between a cubic $X$ and $S$ is essentially an isomorphism of Hodge structures between certain subgroups of their middle cohomologies (see Sect. \ref{general}). A natural conjecture, supported by Hassett's Hodge theoretical work (\cite{Ha1}, \cite{Ha2}, \cite{Ha3}) and Kuznetsov's derived categorical work \cite{Kuznetsov_2009}, is that any rational cubic fourfold ought to have an associated K3 surface. A link between rationality and the transcendental motive of a cubic fourfold has also been described in \cite{BP20}. For now, every fourfold in $\C_{14}$, $\C_{26}$, $\C_{38}$ and $\C_{42}$ has been proved to be rational (see \cite{MR3968870}, \cite{MR3934590}, \cite{russo2019trisecant}). These cubics all have an associated K3 surface.

As it is customary, if $X$ is a smooth cubic fourfold, we will set $A(X):=H^{2,2}(X) \cap H^4(X,\mathbb{Z})$ for the lattice of codimension 2 algebraic cycles up to rational equivalence, and will denote by $\rho_2(X)$ the rank of $A(X)$. More generally, given a projective variety $Y$, $\rho_1(Y)$ will denote the Picard number.

\medskip

In this paper we explore the relation between K3 surfaces and cubic fourfolds under the lens of Chow motives. The first main result is the analogue for cubic fourfolds of a result of the fourth named author, stating that Singular K3 surfaces (i.e. with Picard number equal to 19 or 20)\footnote{of course this terminology is a bit unfortunate but we will use uppercase "S" for this interpretation and lowercase for the "non smooth" meaning} have finite dimensional Chow motive. As usual, we denote by $F(X)$ the Fano variety of lines associated to a cubic fourfold $X$.

\begin{thm*}\nonumber(=Theorem \ref{main})
Let $X\subset \P^5$ be a cubic fourfold with $\rho_2(X)\geq 20$, then the Chow motives $h(X)$ and $h(F(X))$ are finite dimensional.
\end{thm*}

This result is obtained using the description of $F(X)$ as a moduli space of sheaves on the associated K3 (that has Picard number in the range of Pedrini's theorem), combined with a result of Bülles on the motives of moduli spaces.
A similar argument shows that for the same family of cubic fourfolds, the LLSvS hyperKähler eightfold $L(X)$ also has finite dimensional Chow motive (Proposition \ref{8fold}).

This frame of ideas also allows to show a density result for cubic fourfolds with finite dimensional Chow motive inside the moduli space.

\begin{thm*}\nonumber(=Theorem \ref{densekimura})
Let $d$ be not divisible by 4, 9, or any odd prime number $p \equiv 2\ [3]$. Then there exists a dense (in the complex topology) set of points in a non-empty Zariski open subset inside $\mathcal{C}_d$ such that the corresponding fourfolds have finite dimensional Chow motive.
\end{thm*}

In Section \ref{sect 5} we proceed our study of motives of hyperK\"ahler varieties. By combining results from  Laza-Saccà-Voisin, Floccari-Fu-Zhang and some observations about the topology of $\mathcal{C}_{14}$ from \cite{MR3968870}, we show the existence of an infinity of examples of Laza-Saccà-Voisin hyperKähler 10-folds with finite dimensional and Abelian Chow motive. Adding to the picture some recent results of Li-Pertusi-Zhao \cite{pertusietc} about stability conditions on the Kuznetsov component of the derived category of cubic fourfolds, we manage to show the same result for the twisted version of the hyperK\"ahler 10fold, introduced by Voisin \cite{twisted}. For the twisted version we find an infinite number of examples in each divisor $\C_d$ where cubic fourfolds have associated K3 surfaces.

\smallskip

In the two last sections of the paper, we somehow do reverse engineering. Starting from two explicit examples of families $\mathcal{F}$ and $\mathcal{G}$ of K3 surfaces, we manage to construct two families of Fano fourfolds, whose motives have remarkable properties and are finite dimensional, in an appropriate sense. 
 
 The Fano fourfolds from the first family are smooth, notably they are quadric surface fibrations over $\P^2$. The relation with K3 surfaces is given by the Chow-K\"unneth decomposition of the motive of a quadric fibration due to Vial \cite{vialfib}, and some results of Kuznetsov and C\u{a}ld\u{a}raru. The corresponding K3 surfaces are in fact a 4-dimensional family $\mathcal{F}$ of smooth octic surfaces in $\P^5$, that are complete intersections of 3 quadrics (with particularly simple equations) in $\P^5$. Notably $Q_1,Q_2,Q_3$ are hypersurfaces in $\P^5$ \cite[10.2]{GS}, whose equations are

\begin{equation*}
\sum_{0\le i\le 5} a_ix^2_i=0  ;   \sum_{0\le i\le 5} b_ix^2_i=0  \ ; \ \sum_{0\le i\le 5} c_ix^2_i=0,
\end{equation*}

with complex parameters $a_i, b_i, c_i  $  and  $i=0,\cdots,5$. See Sect. \ref{sezquad} for details.
 
\begin{thm*}(=Theorem \ref{3quadrics})
Let $S$ be a general smooth K3 surface in $\mathcal{F}$ and let $Q_1,Q_2$ and $Q_3$ be the three quadrics in $\P^5$ cutting out $S$. Let $\mathcal{Q}'\to \P^2$ be the quadric surface fibration obtained as hyperbolic reduction of the quadric fourfold fibration $\mathcal{Q}\to \P^2$ defined by the quadrics $Q_i$.  Then the Chow motive $h(\mathcal{Q}')$ is finite dimensional and contains the transcendental motive $t(S)(1)$.
\end{thm*}

Hyperbolic reduction is a classical way to obtain a quadric fibration of relative dimension $n-2$, starting from a quadric fibration of relative dimension $n$ with a section. See Sect. \ref{sect 6} for more details. We observe moreover that the finite dimensionality of the Chow motive of $\mathcal{Q}'$ in relation with the transcendental motive of the K3 surface $S$ is a situation similar to that of other Fano fourfolds, starting from cubic fourfolds \cite{BP20}.

\medskip

The fourfolds from the second family are singular cubics, hence rational. For any cubic $X$ of this family, there exists in fact a birational map $\psi_X:\P^4 \dashrightarrow X$ given by the full linear system of cubics through one 15-nodal, sextic K3 surface in $\P^4$ (see Sect. \ref{famnodal}). Let us denote by $\mathcal{G}$ the family of these sextic surfaces; it is not hard to see that it is 4-dimensional. The motive of a desingularization of such a surface has been recently proven to be finite dimensional and of Abelian type \cite{ILP}. This, combined with some not difficult birational transformations, allows us to show that the motive of the corresponding cubic fourfold is Schur-finite inside Voevodsky's triangulated category of motives $\mathbf{DM_Q}(\C)$. 

\begin{prop*}(=Proposition \ref{isolated})
Let $S$ be any surface from the family $\mathcal{G}$, and $X$ the cubic fourfold obtained via $\psi_X$, then $X$ has Schur finite motive in $\mathbf{DM_\Q}(\C)$, belonging to the subcategory ${\bf DM}^{Ab}_{\Q}$, generated by the motives of curves. 
\end{prop*}

Special families of cubic fourfolds with the same property have been  described  by the third named author in \cite{Lat1} and \cite{Lat2}

\smallskip
\textbf{Plan of the paper:} In Section \ref{general}, we recall some generalities about the moduli space of cubic fourfolds and associated K3 surfaces. In Section \ref{cubicmotives},  we start by recalling some results about the Chow motives of cubic fourfolds and K3 surfaces. Then, we prove Theorem \ref{main} and a couple of similar results about hyperK\"ahler varieties. In Sect. \ref{sect 5}, we address the finite dimensionality and Abelianity of the Chow motives of hyperK\"ahler 10folds associated to the twisted Jacobian fibration of cubic fourfolds. In Sect. \ref{sect 6} we recall that every K3 surface $S$ in the first family $\sF$ has a motive of Abelian type and there is a Kuga-Satake correspondence between $h(S)$ and the motive of a Prym variety  $P$ of dimension 4. Then we describe the second family $\sG$ of K3 surfaces $S$  in $\P^4$ with 15 nodes and show that the motive of a desingularization of $S$ is of Abelian type. In Sect. \ref{sect. 7} we show how to construct Fano fourfolds from K3 surfaces from the families $\mathcal{F}$ and $\mathcal{G}$, and show that their motives are of Abelian type as well.

\smallskip
We would like to thank the anonymous referee for pointing out a mistake in a previous version of the paper.
We would like to thank Asher Auel, Igor Dolgachev, Lie Fu, Alice Garbagnati, Bert Van Geemen, Michael Hoff,  Giulia Saccà, Paolo Stellari and Charles Vial for very appreciated conversations and suggestions about the topics of this paper.

\section{Generalities on cubic fourfolds}\label{general}

\subsection{Moduli spaces of cubic fourfolds and K3 surfaces}\label{scubic}

A cubic fourfold $X$ is a smooth complex cubic hypersurface in $\P^5$. The coarse moduli space of cubic fourfolds $\mathcal{C}$ is a 20-dimensional quasi-projective variety. It can be described as a GIT quotient $\mathcal{C}:= \mathcal{U} // PGL(6,\mathbf{C})$, where $\mathcal{U}$ is the Zariski open subset of $\vert \mathcal{O}_{\P^5}(3) \vert$ parametrizing smooth cubic hypersurfaces in $\P^5$.

 Hassett (\cite{Ha1}, \cite{Ha2}) studied cubic fourfolds via Hodge theory and introduced the notion of \textit{special} cubic fourfolds, that is those containing an algebraic surface whose cohomology class is linearly independent 
of $h^2$, where $h$ is the class of a hyperplane section. Let $A(X):=H^4(X,\Z) \cap H^{2,2}(X)$ be the positive definite lattice of integral middle Hodge classes, that coincides with $CH_2(X)$, the Chow group of 2-cycles on $X$. A cubic $X$ is \textit{special} if and only if the rank of $A(X)$ is at least 2. 

\begin{defi}
A labelling of a \textit{special} cubic fourfold is a rank 2 saturated sublattice $K_d \subseteq A(X)$ containing $h^2$. 
Its discriminant $d$ is the determinant of the intersection form on $K_d$.


\end{defi}
\medskip

Special cubic fourfolds  with labelling of discriminant $d$ form a countably infinite union of divisors $\mathcal{C}_d \subset \mathcal{C}$, called Hassett divisors. Hassett \cite[Theorem 1.0.1]{Ha1} showed that $\mathcal{C}_d$ is irreducible and nonempty if and only if 

\begin{center}
$(*)\ \ \ \ d \geq 8$ and $d \equiv 0,2$ $[6]$ \label{*}.
\end{center}

Moreover, in certain cases, one can associate to a cubic fourfold a K3 surface. More precisely, there exists a polarized K3 surface $S$ of degree $d$ such that $K_d^{\perp} \subset H^4(X,\Z)$ is Hodge-isometric to $H^2_{prim}(S,\Z)(-1)$ if and only if 

\begin{center}
$(*')$ \ \ \ \ $d$ is not divisible by 4, 9, or any odd prime number $p \equiv 2\ [3]$  \label{*'}.
\end{center}

Recall that $\rho_1(S)=rk(NS(S))$ and $\rho_2(X)=rk(A(X))$. Whenever $X$ has an associated K3 surface $S$, we have that $\rho_1(S)=\rho_2(X)-1$. 
 
For infinitely many values of $d$, and for the generic cubic fourfold $X\in\mathcal{C}_d$, the Fano variety $F(X)$ of lines on the cubic fourfold is isomorphic to the Hilbert scheme of length two subschemes $S^{[2]}$ of the associated K3 surface $S$. This holds if $d=2(n^2+n+1)$ for an integer $n\geq2$.

\section{Chow motives of cubic fourfolds and K3 surfaces}\label{cubicmotives}

\renewcommand{\L}{\mathbf{L}}



 
The first result we need to recall from \cite{BP20} is the existence of a Chow-K\"unneth  decomposition for the motive of the cubic fourfold $X$. Namely, we have

\begin{equation}\label{CK}
h(X) = \un \oplus \L \oplus (\L^2)^{\rho_2(X)}\oplus t(X)\oplus \L^3 \oplus \L^4, 
\end{equation}

 where $\L$ is the Lefschetz motive and $t(X)$ is the transcendental motive of $X$, i.e. the realization functor gives $H^*(t(X)) =H^4_{tr}(X,\Q)$. 

\medskip

HyperKähler varieties related to cubic fourfolds have descriptions as moduli spaces of sheaves on K3 surfaces, in some different ways. One of them is the following.

\begin{thm}\label{thmprieto}\cite[Theorem 1]{prieto}
Let $Y$ be a projective hyperKähler manifold of $K3^{[n]}$-type, $n \geq 2$,
of Picard number at least four. Then, $Y$ is isomorphic to some moduli space $M_H(S,\alpha,w)$
of twisted stable sheaves on a K3 surface, where $H$ is a polarization on the K3 surface $S$, $\alpha$ is a Brauer class, and $w$ a Mukai vector.
\end{thm}

In particular, we will apply this Theorem to $Y=F(X)$, the Fano variety of lines of a cubic fourfold $X$. The finite dimensionality of the motive of the moduli space will be estimated via the following result of Bülles.

\begin{thm}\label{bullesmoduli}\cite[Theorem 0.1]{Bu}
Let $S$ be a projective K3 surface or an abelian surface and $\alpha \in Br(S)$.
Assume that $M$ is one of the following:

\begin{itemize}

\item a smooth projective moduli space of Gieseker stable $\alpha$-twisted sheaves, or
\item a smooth projective moduli space of $\sigma$-stable objects in $\mathbf{D}^b(S,\alpha)$, where $\sigma$ is a generic
stability condition.

\end{itemize}

Then the Chow motive $h(M)$ of $M$ is a direct summand of a motive $\bigoplus h(S^{k_i})(n_i)$ for some
$1 \leq k_i \leq \dim (M),$ $n_i \in \mathbb{Z}$.
\end{thm}

Let now $X$ be a special cubic fourfold contained in a divisor $\mathcal{C}_d$. In \cite{Bu}, B\"ulles shows that for certain values of $d$, there exists a K3 surface $S$ such that 

\begin{equation}\label{motk3}
t(X) \simeq t_2(S)(1).
\end{equation}\\
Here $t_2(S)$ is the transcendental motive of $S$ , i.e.

$$h(S)=\un \oplus \L^{\rho_1(S)}\oplus t_2(S) \oplus \L^2,$$\\
where $\rho_1(S)$ is the rank of the N\'eron-Severi group $NS(S)$. More precisely, the isomorphism (\ref{motk3}) holds whenever $d$ satisfies the following numerical condition

$$(***)\ \ \ \  \exists  f,g \in \mathbb{Z}  \  with \  g | (2n^2+2n +2)   \  n\in \mathbb{N}  \ and \   d =f^2g.$$

Therefore, in this case, $h(X) \in \M^{Ab}_{rat}(\mathbf{C})$ if and only if $h(S)\in \M^{Ab}_{rat}(\mathbf{C})$. Note that $(*')$ implies $(***)$ and that an isomorphism $t(X)\simeq t_2(S)(1)$ can never hold if $X$ is not special, i.e. if $\rho_2(X)=1$, see \cite[Proposition 3.4]{BP20}. 
 
\medskip

On the other hand, finite dimensionality of motives of K3 surfaces has been addressed in \cite{Ped}. In particular the following is proved:

\begin{thm}\label{k3motive}
Let $S$ be a smooth complex projective K3 surface with $\rho_1(S) = 19, 20$.
Then the motive $h(S) \in \mathcal{M}_{rat} (\mathbf{C})$ is finite dimensional and lies in the subcategory
$\mathcal{M}_{rat}^{Ab} (\mathbf{C})$. 
\end{thm}

We will abuse slightly of terminology by calling \it Abelian \rm a Chow motive that lives in the subcategory $\mathcal{M}_{rat}^{Ab} (\mathbf{C})$.
The last Theorem we need to recall before passing to proving our first main result is the following, that relates the motive of a cubic hypersurface and that of its Fano variety of lines. This is \cite[Theorem 3]{FLV}.

\begin{thm}\label{fanoandcubic}
Let $Y$ be a smooth cubic hypersurface in $\P^{n+1}$ and $F(Y)$ the associated Fano variety of lines on $Y$. We have an isomorphism of Chow motives

\begin{equation}\label{FLViso}
h(F(Y))\cong \Sym^2(h^n(Y)_{prim}(1)) \oplus \bigoplus_{i=1}^{n-1}h^n(Y)_{prim}(2-i)\oplus \bigoplus_{k=0}^{2n-4}\un (-k)^{\oplus a_k},
\end{equation}

where 

\begin{equation}
a_k= \begin{cases}
  \lfloor \frac{k+2}{2} \rfloor & \text{if } k < n-2 \\
  \lfloor \frac{n-2}{2} \rfloor  & \text{if } k= n-2 \\
  \lfloor \frac{2n -2 -k}{2} \rfloor & \text{if } k> n-2. \\
\end{cases}
\end{equation}

\end{thm}


\subsection{Cubic fourfolds with finite dimensional motive.} 

The goal of this section is proving our first main result.

\begin{thm}\label{main}
Let $X\subset \P^5$ a cubic fourfold with $\rho_2(X)\geq 20$, then the Chow motives $h(X)$ and $h(F(X))$ are finite dimensional.
\end{thm}

\begin{proof}
Let $F(X)$ be the Fano variety of lines of $X$. By the Abel-Jacobi isomorphism $H^4(X,\mathbb{Z})_{prim}\cong H^2(F(X),\mathbb{Z})_{prim}(-1)$ we get that the Picard number $\rho_1(F(X))$ is equal to $\rho_2(X)\geq 20$. Hence, by Theorem \ref{thmprieto}, $F(X)$ is isomorphic to a certain moduli space $M_v(S,\alpha)$ of twisted stable sheaves on a K3 surface $S$ (where $v$ is the Mukai vector, and $\alpha$ a Brauer class).  More precisely, in \cite{prieto}, the author observes that a combination of a result by Morrison \cite[Cor. 2.10]{Morrison84} with Theorem \ref{thmprieto} implies that, whenever $\rho_1(F(X))\geq 13$, $F(X)$ is isomorphic to a moduli space of (non-twisted) stable sheaves $M_H(S,v)$ on a polarized K3 surface $(S,H)$, with Mukai vector $v$. 
This is the case in our situation and, by Theorem 1 of \cite{add2conj}, it directly implies that $S$ is the K3 surface associated to $X$.
On the other hand, by Theorem \ref{bullesmoduli}, we have that $h(M_H(S,v))$ is a direct summand of $\bigoplus_i (S^{k_i})(n_i)$, for some $1 \leq k_i \leq \dim (M)$, and $n_i \in \mathbb{Z}$. Since $S$ is the K3 surface associated to $X$, the Picard number $\rho_1(S)$ is equal to $\rho_2(X)-1$, which means $\rho_1(S) \geq 19$. By Theorem \ref{k3motive}, we know that K3 surfaces with $\rho_1(S)\geq 19$ have finite dimensional and Abelian Chow motive, and by a classical result of Kimura \cite{kimurafinite}, the Chow motive of powers of $S$ is finite dimensional too. This in turn implies that $h(M_H(S,v))=h(F(X))$ is finite dimensional. Now, by Equation \ref{FLViso}, we get that also the $h^4(X)_{prim}$ is finite dimensional and Abelian. This contains the transcendental motive $t(X)$, which by Equation \ref{CK} is possibly the only non finite dimensional summand in the Chow-Künneth decomposition. This concludes the proof.

\end{proof}

By results of \cite{BFMQ}, when $F(X)$ is the variety of lines of a cubic $X\in\mathcal{C}_{12}$, then it is birational to two EPW double sextics $E_{F(X)}$ and $E'_{F(X)}$. Since birational hyperKähler varieties have isomorphic Chow motives \cite{riess} , this automatically gives us the following corollary.

\begin{cor}
Let $X\in \mathcal{C}_{12}$ with $\rho_2(X)>19$, then the Chow motives of $E_{F(X)}$ and $E'_{F(X)}$ are finite dimensional and Abelian.
\end{cor}

\subsection{Density of cubic fourfolds with finite dimensional motive}\label{density}

More generally, let $\mathcal{G}_d$ be the moduli space of polarized K3 surfaces of degree d. This is a quasi-projective 19-dimensional algebraic variety.

\begin{thm}\label{densekimura}
Let $d$ be not divisible by 4, 9, or any odd prime number $p \equiv 2\ [3]$. Then there exists a dense (in the complex topology) set of points in a non-empty Zariski open subset inside $\mathcal{C}_d$ such that the corresponding fourfolds have finite dimensional Chow motive.
\end{thm}

\begin{proof}
Let $X\in \mathcal{C}_d$, for $d$ in the range of the claim here above. That is: $X$ has one (or two, see \cite{Ha1}) associated polarized K3 surface $S_X$ in $\mathcal{G}_{\frac{d+2}{2}}$. Then the map

\begin{eqnarray}\label{asso}
\mathcal{G}_{\frac{d+2}{2}} & \to & \mathcal{C}_d;\\
S_X & \mapsto & X;
\end{eqnarray}\\
is rational and dominant. Hence, if $d$ is in the range here above, there exists an open set $\mathcal{U}_d$ of $\mathcal{C}_d$ such that for every $X\in \mathcal{U}_d$ there exists a K3 surface $S_X$ of degree $d$ associated to $X$. We observe also that Singular K3 surfaces form a  subset of the moduli space $\mathcal{G}_{\frac{d+2}{2}}$ which is dense in the complex topology. The proof of this fact goes along the same lines as the proof of the density of all K3 surfaces in the period domain (see  \cite[Corollary VIII.8.5]{ccs}). By the dominance of the map in $(\ref{asso})$, this directly implies the claim.
\end{proof}

\subsection{Some remarks on hyperK\"ahler varieties}

In this section, we draw consequences on Abelianity and finite dimensionality of the motive of some hyperK\"ahler varieties related to cubic fourfolds from the preceding results. As we have seen in the preceding section, the Chow motive of the Fano variety of $X$ is finite dimensional whenever $\rho_2(X)\geq 20$.

\medskip

Let $L(X)$ denote the 8-fold constructed in \cite{LLSvS} from the space of twisted cubic curves on a cubic fourfold not containing a plane. The 4-dimensional $F(X)$ is deformation equivalent to the Hilbert scheme  $S^{[2]}$, with $S$ a K3 surface, while $L(X)$ is deformation equivalent to $S^{[4]}$. 
For every even complex dimension $2n$ there are two known deformations classes of irreducible holomorphic symplectic varieties: the Hilbert scheme $S^{[n]}$ of n-points on a K3 surface $S$ and the generalized Kummer. A generalized Kummer variety $Y$ is of the form $Y=K^n(A)=a^{-1}(0)$, where $A$ is an Abelian surface and $a:  A^{[n+1]} \to A$ is the Albanese map. 
In dimension 10 there is also an example, usually referred as OG10,  discovered by O'Grady. The hyperK\"ahler variety OG10 is not deformation equivalent to $S^{[5]}$.

\medskip

Let  $\M_A(\mathbf{C})$ be  the category of Andr\'e motives which is obtained from the category of homological motives $\M_{hom}(\mathbf{C})$ by formally adjoining the Lefschetz involutions $*_L$ associated to the  Lefschetz  isomorphisms $L^{d-i}: H^i(X) \to H^{2d-i}(X)$, where $L^{d-i}$ is induced by the hyperplane section. By the Standard Conjecture $B(X)$, for every $i\le d$ there exists an algebraic correspondence inducing the isomorphism $H^{2d-i}(X)\to H(X)$ inverse to $L^{d-i}$. Therefore, under $B(X)$ the category of Andr\'e motives coincides with $\M_{hom}(\mathbf{C})$. The Andr\'e motive of a K3  surface $S$  and of  a cubic fourfold $X$ belong to the  full subcategory  $\M^{Ab}_{A}(\mathbf{C}) $ generated by the motives of Abelian varieties, see \cite[10.2.4.1]{An}.

\medskip

In \cite{Sc} it is proved that the Andr\'e motive of a hyperK\"ahler variety which is  deformation equivalent  to  $S^{[n]}$ lies in $\M^{Ab}_A(\mathbf{C})$.  Soldatenkov \cite{So} proves that  if $X_1$ and $X_2$ are  deformation equivalent projective hyperK\"ahler manifolds then the Andr\'e motive of $X_1$  is Abelian if and only if the Andr\'e motive of $X_2$ is Abelian. In a recent preprint (see \cite{FFZ})  it is proved that also the Andr\'e motive of OG10 lies in $\M^{Ab}_{A}(\mathbf{C})$. Therefore the Andr\'e motives of all the known deformation classes of hyperK\"ahler varieties lie in $\M^{Ab}_A(\mathbf{C})$. These results suggest the following conjecture:

\begin{conj}  The motive of a hyperK\"ahler manifold  is of Abelian type in $\M_{rat}(\mathbf{C})$.\end{conj}




Let us now consider the hyperK\"ahler 8fold $L(X)$. In order to define properly $L(X)$, we need to assume that $X$ does not contain a plane, i.e. $X\not\in\mathcal{C}_8$. Then, the analogue of Theorem \ref{main} is the following:

\begin{prop}\label{8fold}
Let $X\subset \P^5$ a cubic fourfold with $\rho_2(X)\geq 20$.
Then the hyperK\"ahler 8fold $L(X)$ has finite dimensional and Abelian Chow motive.
\end{prop}

\begin{proof}
By the results in \cite{BLMS} and \cite{LPZ} we know that $L(X)$ is a moduli space of stable objects in $\mathcal{K}u(X)$ for a certain Mukai vector. This also implies that $\rho_1(L(X))=\rho_2(X)-1\geq 19$. Since $L(X)$ is a hyperKähler variety of $K3^{[4]}$-type, Theorem \ref{thmprieto} implies that $L(X)$ is in particular a moduli space $M_H(S,v)$ of stable sheaves over a polarized K3 surface. The same argument as in the case of $F(X)$ via the results of Bülles and the associated (Singular) K3 surface implies that the Chow motive $h(L(X))$ is finite dimensional.





%
\end{proof}


\section{A remark on the finite dimensionality of the motive of the LSV 10 dimensional hyperKähler manifold} \label{sect 5}

\subsection{The intermediate Jacobian fibration}

\newcommand{\cJ}{\mathcal{J}}

Let $X\subset \P^5$ be a smooth cubic fourfold, and $\cJ(X)$ the 10 dimensional hyperKähler manifold constructed in \cite{LSV}. The variety $\cJ(X)$ is a compactification of the fibration over $\P^{5*}$, whose fibers are the intermediate Jacobians of the hyperplane sections of $X$, and it is deformation equivalent to O'Grady ten dimensional example OG10 \cite[Cor. 6.3]{LSV}. 

\smallskip

Let $X$ be a Pfaffian cubic fourfold. Recall from a lot of references (\cite{Beau,kuz_levico,Ha1}) that the K3 surface associated to a Pfaffian cubic is the "orthogonal"  section of the Grassmannian $G(2,6)$. Let $W$ be a 6 dimensional vector space. If $X$ is a pfaffian cubic defined by choosing a 6 dimensional vector space $V$ and an embedding $\P(V)\subset \P(\wedge^2 W^*)$, then we can consider the 9 dimensional annihilator $V^\perp$ in $\wedge^2(W)$. The projectivized $\P^8$ cuts out a degree 14 K3 $S$ from the $G(2,6)$ that naturally lives in $\P(\wedge^2 W)$, and $S$ is the K3 associated to $X$. 

\smallskip

An argument similar to those appearing in the preceding section will allow us to show that there exist several examples of $\cJ(X)$ with finite dimensional, Abelian Chow motive. First of all we need to recall the following result \cite[Theorem. 6.2]{LSV}.

\begin{thm}\label{birat}
If $X$ is a smooth Pfaffian cubic fourfold, then $\cJ(X)$ is birational to the
O'Grady moduli space $\mathcal{M}_{2,0,4}(S)$ parameterizing rank-2 semi-stable sheaves on $S$ with $c_1=0$ and $c_2=4$.
\end{thm}

Recall that O'Grady hyperKähler example OG10 is a birational desingularization of $\mathcal{M}_{2,0,4}(S)$. The main result of this section is the following.

\begin{thm}
There are infinitely many Pfaffian cubic fourfolds $X\in \mathcal{C}_{14}$ such that that the motive $h(\cJ(X))$ of the associated LSV 10-fold is finite dimensional and Abelian.
\end{thm}

\begin{proof}
Recall from \cite[Remark 4.8]{MR3968870} that the locus of Pfaffian cubics inside $\mathcal{C}_{14}$ is a constructible set. This means that it contains open subset, that we will denote by $U$.
On the other hand, we know \cite{Ha1} that there is a birational map

$$g:\mathcal{C}_{14} \dashrightarrow \mathcal{G}_8$$

sending a cubic fourfold onto its associated K3 surface. This means that there exists an open subset $A\subset \mathcal{C}_{14}$, which is isomorphically sent onto an open set $g(A)\subset \mathcal{G}_8$. Let us consider the intersection $A\cap U\subset \mathcal{C}_{14}$. The image $g(A\cap U)$ is still open in $\mathcal{G}_8$ and all K3 surfaces in $g(A\cap U)$ are associated to a smooth Pfaffian cubic fourfold.
Now, it is well known that K3 surfaces with $rk(NS)\geq 19$ are dense (in the complex topology) in the 19-dimensional moduli space $\mathcal{G}_d$ of degree $d$ polarized K3 surfaces, for all $d$. This means that there are also K3 surfaces of this kind inside our open subset $g(A\cap U)$. It is also straightforward to see that there are infinitely many, by a density argument. In turn, this means that there exist Pfaffian cubic fourfolds, whose associated K3 has Néron-Severi group of rank at least 19. Hence, by \cite{Ped}, they have a finite dimensional and Abelian Chow motive.

\smallskip

The existence of the LSV HK 10fold, compactifying the intermediate Jacobian fibration, is verified in \cite{LSV} only for the generic cubic fourfold. More recently, Saccà \cite{sacca} has shown the existence of the HK compactification for all cubics, hence it is in this framework that we will place ourselves.
Now, by \cite[Cor. 4.5 and 4.7]{FFZ}, we know that the Chow motive of the OG10 hyperKähler associated to a K3 surface S is finite dimensional and Abelian whenever the $h(S)$ is. Then, if the cubic fourfold $X$ is Pfaffian, by Theorem \ref{birat} $\cJ(X)$ is birational to the OG10 associated to its K3 $S$. This implies that over the open subset $A \cap U$ there are infinitely many $\cJ(X)$ that, due to their birationality to the corresponding OG10, have finite dimensional and Abelian motive.
\end{proof}


\subsection{The twisted jacobian fibration}

\newcommand{\barj}{\overline{\cJ}(X)}

This section is the development of a suggestion of Giulia Saccà. In the last few years, Voisin \cite{twisted,voisinaj} introduced a twisted version of the intermediate jacobian fibration (see \cite[Sect. 3]{twisted} for details), parametrizing 1-cycles of degree 1 in the fibers of the cubic 3fold fibration. By the results of Voisin \cite{twisted} and Saccà \cite{sacca}, it is now known that, for any cubic fourfold $X$, there exist a hyperK\"ahler compactification of the twisted intermediate Jacobian fibration. The main result of this section shows that even in the twisted case, the motive of these HK compactifications can be often finite dimensional.

\begin{thm}
For $d$ not divisible by 4, 9, or any odd prime number $p \equiv 2 [3]$, there
exists an infinity of cubic fourfolds in $\C_d$ s.t. the twisted Jacobian fibration has a HK compactification with finite dimensional Chow motive.
\end{thm}

\begin{proof}
We observe that the condition on $d$ is nothing but the fact that the cubic fourfold $X$ has an associated K3 surface. Let $Ku(X)$ be the Kuznetsov component of the derived category of the cubic 4fold $X$. In the recent paper \cite{pertusietc}, the authors studied the moduli space $M(X)$ of semistable objects of Mukai vector $2\lambda_1+2\lambda_2$ in $Ku(X)$. Recall in fact from \cite{AT} that the algebraic Mukai lattice of $Ku(X)$ contains two distinguished classes $\lambda_i$, $i=1,2$, that span an $A_2$-lattice. With an appropriate choice of a stability condition on $Ku(X)$ (that we will not detail here, see \cite[Sect. 2 and 3]{pertusietc} for details), the authors manage to show that the moduli space $M(X)$ has a symplectic resolution $\widetilde{M}(X)$, which is a projective hyperK\"ahler manifold, deformation equivalent to O'Grady 10 dimensional examples. More importantly, they show that there exists a HK birational model $N(X)$ of $\widetilde{M}(X)$ that  provides a compactification of the twisted intermediate Jacobian fibration. Of course $N(X)$ and $\widetilde{M}(X)$ have isomorphic Chow motive, since they are birational.

\medskip

The same argument as in \cite[Section 5]{FFZ} (see also \cite[Theorem 1.3 and Remark 1.4]{FLV}) shows that $\widetilde{M}(X)$ has finite dimensional Chow motive whenever $h(X)$ is finite dimensional. This means once again that, if $X$ has an associated Singular K3 surface, then the motive $h(N(X))$ of the HK compactified twisted Jacobian fibration is finite dimensional. Since singular K3 surfaces are dense in the moduli space of genus $\frac{d+2}{2}$ K3 surfaces, by the same argument as in the untwisted case, this completes the proof.

\end{proof}

\section {K3 surfaces with Abelian Chow motive}
\label{sect 6}

In this section, we construct higher dimensional families of cubics with Abelian Chow motive starting from families of K3 surfaces with particular shapes. It is well known that among the examples of K3 surfaces with a motive of Abelian type there are Kummer surfaces and K3 surfaces with Picard rank $\ge 19$ (see \cite{Ped}). In this section we consider  two families $\sF$ and $\sG$ of K3 surfaces with Picard rank 16, constructed in \cite{GS}.\par

 In \cite{GS} the authors construct  families  of K3 surfaces $S$, with $\rho(S)=16$, admitting a symplectic action of  the group $ G =(\Z_2)^4$. The following lemma shows that  the transcendental motive $t_2(S)$, as defined in \cite{KMP}, only depends on the motive of the quotient surface $S/G$. Note that
 the motive $h(S/G)$ can be represented in $\sM_{rat}(\mathbf{C})$ see \cite[16.1.13]{Fu}.

\begin{lem}\label{symplectic} Let $S$ be a K3 surface over $\mathbf{C}$  with a finite group $G$ of symplectic 
automorphisms.  Let $Y$ be a  minimal desingularization of the quotient surface $S/G$. Then

$$t_2(S)\simeq t_2(Y)$$ 

in $\sM_{rat}(\mathbf{C})$.
\end{lem}

 \begin {proof}  From the results in \cite{Huy2},  every symplectic automorphism 
$g \in G$ acts trivially on $A_0(S)$, so that $A_0(S)^G =A_0(S)$.
 Since $G$ is symplectic there are a finite number
 $\{P_1,\cdots,P_k\}$ of isolated fixed points for $G$ on $S$. Let $Y$
 be a minimal desingularization of the quotient surface $S/G$. The
 maps $f : S \to S/G$ and $g : Y \to S/G$ yield a rational map $S\dashrightarrow Y$
 which is defined outside $\{P_1,\cdots,P_k\}$. Since  $t_2(-)$ is a birational invariant for smooth projective surfaces we 
 get a map $\theta : t_2(S) \to t_2(Y)$ such that $\theta$ is the map onto a direct summand, see \cite[Proposition 1]{Ped}. Hence we can write
 
$$t_2(Y)=t_2(S)\oplus N.$$

Since $A_i(t_2(S))=0$, for $ i \ne 0$, and $A_0(S)^G =A_0(S)$, we get $A_i(N)=0$, for all $ i$. From \cite[Lemma 1]{GG} we get  
$N=0$, hence  $t_2(S) =t_2(S)^G \simeq t_2(Y)$.\par
\end{proof} 
 
\subsection{Intersection of three quadrics}\label{sezquad}

\noindent The moduli space of  K3 surfaces with Picard number 16, admitting a symplectic action of the group $G =(\Z_2)^4$, has a countable number of connected components of dimension 4. The generic element of one of these connected components, that we will denote by $\sF$, can be realized by considering a complete intersection $S$ of three quadrics  $Q_1,Q_2,Q_3$ in $\P^5$, see \cite[10.2]{GS}, whose equations are

\begin{equation}\label{quadriche}
\sum_{0\le i\le 5} a_ix^2_i=0  ;   \sum_{0\le i\le 5} b_ix^2_i=0  \ ; \ \sum_{0\le i\le 5} c_ix^2_i=0,
\end{equation}

with complex parameters $a_i, b_i, c_i  $  and  $i=0,\cdots,5$. The group $G$  is realized as the
transformations of $\P^5$ changing an even number of signs in the coordinates  $\{x_0: x_1: x_2: x_3 : x_4 : x_5\}$. Let us denote by $\sF$ the family of such K3 surfaces.
The dimension of the moduli space of these K3 surfaces is 4, see \cite[10.2]{GS}.



By Lemma \ref{symplectic} we get  $t_2(S)=t_2(Y)$, with $Y$ a desingularization of $S/G$. An important feature of these surfaces is described in \cite[Theorem 3.1]{Lat3}.

\begin{thm}\label{Lato3} A K3 surface belonging to the family $\sF$ has a motive of Abelian type.
\end{thm}

More precisely one can show, using the results in \cite{Para}, that the motive $h(S) $ belongs to the subcategory of 
$\mathcal{M}^{ab}_{rat}$ generated by the motive of a Prym variety associated to the surface $Y$. By \cite[10.2]{GS} the quotient $S/G$ is a double cover of $\P^2$ branched at six lines $l_i$ meeting in 15 points.  The vector space $NS(Y)\otimes \Q$  of the desingularization $Y$ of $S/G$  is generated by the 15 classes $E_{i,j}$ of the 15 exceptional curves over the intersection points $l_i\cap l_j$ and by the class $h$ of the inverse image of a general line in $\P^2$. The motives of the surfaces $S$ and $Y$  have a Chow-K\"unneth decomposition as in \cite[7.2.2]{KMP} 

\begin{eqnarray}
h(S) = \un \oplus \L^{\oplus \rho(S)} \oplus t_2(S) \oplus \L^2 ;\\  
h(Y)= \un \oplus \L^{\oplus \rho(Y)}\oplus t_2(Y) \oplus \L^2;
\end{eqnarray} 

\noindent where $\rho(S)=\rho(Y)=16$ and $t_2(Y)=t_2(S)$, by \ref{symplectic}. Therefore $h(Y)=h(S)$.\par
\noindent  By the results in \cite{Para} there exists a surface $W$ which is a desingularization of the quotient $(C \times C)/F$ with  $C$ a curve and  $F$ a finite group such that $Y =W/ i$, where  $i$  is a symplectic involution. The curve $C$ has genus 5 and has an automorphism $h$  of order 4 such that the quotient  $C/h$ is an elliptic curve $E$. The finite group $F$ acting on $C \times C$ is generated by the automorphism $(h,h^{-1})$ and the involution $(c_1,c_2) \to (c_2,c_1)$. Let $P$ be the connected component of the identity in the kernel of the natural homomorphism $\Jac \ C  \to \Pic E$.  The Prym variety $P$ is an Abelian variety of dimension 4  and the Kuga-Satake variety $K(Y)$ is a sum of copies of $P$.

\begin{prop}\label{Prym} The motive $h(S)$ belongs to the subcategory of $\sM^{Ab}_{rat}$ generated by the motive h(P).\end{prop}

\begin{proof}  By the Main Theorem in \cite{Para} there exists an algebraic cycle $\Gamma \in A^2(P \times Y)$ such that the associated map $\Gamma^*: H^2(Y, \Q) \to H^2(P,\Q)$ induces an inclusion between the lattices  of transcendental cycles. Therefore the vector space $H^2_{tr}(Y,\Q) $ is a direct summand of $H^2_{tr}(P,\Q)$. Since $P$ is an Abelian variety its Chow motive $h(P)$ has a C-K decomposition
$$h(P)=\sum_{0 \le i \le 8} h_i(P).$$
\noindent The correspondence $\Gamma \in A_4(P \times Y)$ gives a map $ h(P) \to h(Y)$, in the covariant category $\sM_{rat}(\mathbf{C})$. By composing with the inclusion $h_2(P) \subset h(P)$ and the projection $h(Y) \to t_2(Y)$ we get a map of motives

\begin{equation} \label {split} \gamma  : h_2(P) \to t_2(Y) \end{equation}

\noindent such that the corresponding map at the level of cohomology is split surjective. The motives $h(P)$ and $t_2(Y)=t_2(S)$ lie in the subcategory $\sC\subset \sM_{rat}(\mathbf{C})$ generated by finite dimensional motives. Let $\sN$ be the largest tensor ideal in $\sC$ such that the quotient category $\sC/\sN$ is semi simple. The ideal $\sN$ corresponds to numerical equivalence of cycles, the functor $\sC\to \sC/\sN$ is conservative and reflects split epimorphisms, see \cite[1.4.4 and 8.2.4]{AK}. Therefore 
the map $\gamma$ is split surjective in $\sM_{rat}(\mathbf{C})$. This proves that $t_2(Y)=t_2(S)$ is a direct summand of $h_2(P)$, hence  the motive $h(S)=h(Y)$ lies in the subcategory of $\sM_{rat}^{Ab}$
generated by the motive of the Prym variety $P$.
\end{proof}

\subsection{Degree six surfaces in $\P^4$.} A further example of  a family $\sG$  of K3 surfaces  with a motive of Abelian type is given by degree 6 surfaces $S$ in $\P^4$ with 15 ordinary nodes. The desingularization of these surfaces has Picard rank 16 as well.
By a result in \cite[8.2]{GS} these surfaces are the quotient of a K3 surface $X$  by the action of a symplectic group $G \simeq (\Z_2)^4$ of automorphisms. We have a diagram

\begin{equation} \label{diagram} \CD  \tilde X@>>>X\\
@V{\pi}VV   @V{f}VV  \\
Y@>{g}>> S \\ \endCD \end{equation}

\noindent where $\tilde X$ is the blow-up of the fixed points under the action of $G$ and $Y$ contains 15  rational curves coming from the resolution of the singularities of $S=X/G$. The map  $\pi$ is $16:1$ outside the branch locus.  By Lemma \ref{symplectic}, the maps in \ref{diagram} induce an isomorphism of motives $h(X) \simeq h(Y)$.

The  results  of Paranjape have been recently extended in \cite[Cor. 6.16]{ILP}. 

\begin{thm}\label{15 nodes} Let $S$ be  general K3 surface of degree 6 with 15 ordinary nodes, i.e singularities of type $A_1$. Then the motive of a desingularization of  $S$ is finite dimensional and of Abelian type.\end{thm}

In particular every surface in the family $\sG$ has a motive of Abelian type.

 \section{Fano fourfolds with associated K3 surfaces in $\sF$ or $\sG$}
\label{sect. 7}

In this section we use finite dimensionality of  K3 surfaces belonging to the families  $\sF$ and $\sG$ to construct Fano fourfolds either with a finite dimensional motive, or that contain a significant finite dimensional submotive (actually the motive of a K3 surface). This is similar to the approach of \it categorical representability \rm followed in \cite{BBrepre}. In Theorem \ref{3quadrics} we consider a family of special Fano fourfolds  fibrated in quadric surfaces. On the other hand, in Proposition \ref{isolated} we study singular cubic fourfolds associated to K3 surfaces belonging to $\sG$ . In the second case the motive of the singular cubic  4-fold  is Schur-finite in Voevodsky's triangulated category of motives  $ {\bf DM}_{\Q}(\mathbf{C})$. 

\subsection{Fano fourfolds associated to the net of quadrics of an octic K3 surface.}




Let  $Q_1,\ Q_2$ and $Q_3$ the three quadrics defined in Eq. \ref{quadriche}. They generate a net of quadrics, whose total space is a quadric fibration $\sigma: \mathcal{Q}\to \P^2$ of relative dimension 4, contained in $\P^2\times \P^5$. A simple calculation shows that for generic choices of $a_i, b_i$ and $c_i$ inside Equation \ref{quadriche}, the generic fiber of $\sigma$ is smooth and the singular fibers have an isolated singularity. The net degenerates along a plane sextic, and recall that the degree 8 complete intersection of the $Q_i$ is (generically) a smooth K3 surface $S$, that is the base locus of the net.

We are now going to perform the hyperbolic reduction (see \cite[Sect. 1]{ABB} for details) of the quadric bundle with respect to the constant section given by any one of the points of $S$. The result will be a second quadric fibration, but this time the fibers will be surfaces, and the total space a Fano fourfold.

We observe that $\mathcal{Q}$ is given by a $6\times 6$ symmetric matrix $M_{\mathcal{Q}}$ of linear forms, that define a map $\mathcal{O}_{\P^2}^{\oplus 6} \to \mathcal{O}_{\P^2}(1)^{\oplus 6}$. This means that $\mathcal{Q}$ can be seen as the zero locus of $(\mathcal{O}_{\P^2}^{\oplus 6}, q, \mathcal{O}_{\P^2}(1))$, where by this we mean a line bundle valued quadratic form $q$ on $\mathcal{O}_{\P^2}^{\oplus 6}$, with values in $\mathcal{O}_{\P^2}(1)$. For simplicity, in this section we will denote $E:=\mathcal{O}_{\P^2}^{\oplus 6}$. The K3 surface has points over $\mathbb{C}$ hence we can chose the constant section given by one of these over $\P^2$. This gives an inclusion of an isotropic line sub-bundle $N:=\mathcal{O}_{\P^2}\subset E$. The hyperbolic reduction of $\mathcal{Q}$ naturally lives in the projectivized bundle $N^\perp/N$ which is defined by the following two short exact sequences.

\begin{eqnarray}
0\to  N^\perp \to  E \to   Hom( \mathcal{O}_{\P^2},\mathcal{O}_{\P^2}(1)\to  0; \\
0 \to  N \to  N^\perp \to  N^\perp/N \to  0.
\end{eqnarray}

An easy Chern class computation based on the above SES shows that $c_1(N^\perp/N)= h$, where $h$ is the hyperplane class of $\P^2$. Let us denote by $\sigma': \mathcal{Q}'\to \P^2$ the quadric surface bundle obtained by hyperbolic reduction of $\mathcal{Q}$ at the constant section given by projectivizing $N$.

\begin{lem}
The quadric surface fibration $\sigma ':\mathcal{Q}'\to \P^2$ is a smooth Fano variety.
\end{lem}

\begin{proof}
Hyperbolic reduction preserves the discriminant and the corank of the singular fibers, hence smoothness holds. Let us set $H_1=\sigma^{\prime*}h$ and $H_2=\mathcal{O}_{\P(N^\perp/N)}(1)_{|\mathcal{Q}'}$. In order to compute the anticanonical bundle, we can use the formula of Proposition 2.13 of \cite{alex} and obtain

$$-K_{\mathcal{Q}'}= \frac{12 -2 - 6}{2+2}H_1 + 2H_2.$$

Since $H_2$ is ample, one easily sees that $-K_{\mathcal{Q}'}$ is ample as well and $\mathcal{Q}'$ is Fano.
\end{proof}


\begin{thm}\label{3quadrics}
The Chow motive of the Fano fourfold $\mathcal{Q}'$ is finite dimensional, Abelian and contains the transcendental motive $t(S)(1)$.
\end{thm}

\begin{proof}
 Let $S$ be the K3 surface defined by the intersection of the 3 quadrics $Q_1,Q_2,Q_3$. The quadrics $Q_i$ define a linear system of dimension two $|Q|=\P^2$. This projective plane naturally contains the discriminant curve $D$ of degree 6, that parametrizes singular quadrics, and under our genericity hypotheses the sextic is smooth. The curve $D$ defines another smooth K3 surface $T$, which is 
the double cover of $\P^2$ ramified in $D$.

 
By the work of Kuznetsov \cite{kuzquad} on derived categories of quadric fibrations, we know that there exists a semi-orthogonal decomposition for the bounded derived category of coherent sheaves on $\mathcal{Q}'$:

\begin{equation}\label{dbquad}
D^{\mathbf{b}}(\mathcal{Q}') = \langle D^{\mathbf{b}}(T, \alpha), \sigma'^{*}(D^{\mathbf{b}}(\P^2))\otimes \mathcal{O}_{\mathcal{Q}'/\P^2}(1), \sigma'^{*}(D^{\mathbf{b}}(\P^2))\otimes \mathcal{O}_{\mathcal{Q}'/\P^2}(2) \rangle.
\end{equation}
 
Here $T$ is our K3 double cover of $\P^2$, and $\alpha\in Br[2](T)$ the (2-torsion) Brauer class on $T$ that comes from the Brauer-Severi variety on $T$ defined by the relative Hilbert scheme of lines inside the fibers of $\sigma':\mathcal{Q}'\to \P^2.$ Moreover, thanks to the work of Cald\u{a}r\u{a}ru \cite{calda}  , we know that we have a derived equivalence

\begin{equation}\label{calde}
D^{\mathbf{b}}(T, \alpha)\cong  D^{\mathbf{b}}(S).
\end{equation}

The fully faithful functor $D^{\mathbf{b}}(S) \to D^{\mathbf{b}}(\mathcal{Q}')$ is Fourier-Mukai. A straightforward application of Grothendieck-Riemann-Roch shows that this induces a correspondence
$\gamma \in A^3(S \times \mathcal{Q}')$ that gives a map of motives $h(S) (1) \to   h(\mathcal {Q'}) $.

Recall moreover that we have the decomposition, proved in  Vial  \cite[Cor. 4.4]{vialfib}, 

\begin{equation}\label{ckquad}
h(\mathcal{Q}')\cong \bigoplus_{i=0}^2 h(\P^2)(i) \oplus (Z,r,1),
\end{equation}

because $h(\mathcal{Q}')$ is a quadric bundle over $\P^2$. Here $Z$ is a surface and $r$ a projector.
Now we claim that $A_0(S)_{hom} \cong r_*A_0(Z)_{hom}$. In fact, the equivalence in Equation \ref{calde} implies
$A_0(S)_{hom} \cong A^\ast_{hom}(T)$, and $A^\ast_{hom}(T)$ is in turn isomorphic to $A^\ast_{hom}(\mathcal{Q}^\prime)$, just by taking $A^\ast_{hom}$ on both sides of Equation \ref{dbquad}.
On the other hand, we also note that $A^\ast_{hom}(\mathcal{Q}^\prime)=A^3_{hom}(\mathcal{Q}^\prime)$ is isomorphic to $r_\ast A_0(Z)_{hom}$ by taking $A^\ast_{hom}$ on both sides of Equation $\ref{ckquad}$.





 Since, by Equation \ref{ckquad}, the only possibly non-abelian, non finite dimensional, part of $h(\mathcal {Q'}) $ depends on the motive $(Z,r,1)$, we just need to prove that the map of motives $h(S) \to (Z,r,0)$ induced by $\gamma_*: h(S) (1) \to   h(\mathcal {Q'}) $ gives an isomorphism between the transcendental parts, since the algebraic parts are finite dimensional. The transcendental parts of $h(S)$ and $(Z,r,0)$ are respectively $t_2(S)$ and $r_*(t_2(Z))$, whose Chow groups $A_0(S)_{hom}$ and $r_*A_0(Z)_{hom}$ are isomorphic. Therefore they are isomorphic and the transcendental motive of $S$ is finite dimensional and abelian.
\end{proof}







\medskip



\subsection{Singular cubics and 15-nodal K3 surfaces.}\label{schurbirat}

The relation between smooth cubic fourfolds and polarized K3 surfaces, proved in \cite{Ha1}, has been recently extended to the case of fourfolds with isolated ADE singularities. A.K. Stegmann in 
\cite{Steg} constructs the moduli space of cubic fourfolds with a certain combination of isolated ADE singularities as a GIT quotient and compares it to the moduli space of certain quasi-polarized K3 surface of degree 6, proving that the two moduli spaces are isomorphic. Here we consider the motive of  a cubic fourfold  $X$  with isolated  singularities associated to  a K3 surfaces  of degree 6 with 15 nodes belonging to $\sG$. Since we are now working with singular varieties, we need to change the category of motives from $\sM_{rat}(\mathbf{C})$ to   Voevodsky's triangulated category of motives  ${\bf DM}_{\Q}(\mathbf{C})$.\par

Let $X\subset \P^5$ be a cubic fourfold with isolated singularities. Projection from a double point $p \in X$  gives a birational map $\pi_p : X \dashrightarrow \P^4$ which can be factored as

$$ \CD \tilde X_p@>{q_1}>>X   \ ;   \       \tilde X_p@>{q_2}>>\P^4 \endCD$$
where $q_1$ is the blow-up of the singular  point $p$ and $q_2$ is the blow-down of
the lines contained in $X$  passing through $p$.  These lines are parametrized by a normal surface $S_p$ of degree 6 in $\P^4$. The surface $S_p$ is a complete intersection of  a quadric $Q_p$ and a cubic $C_p$. The quadric $Q_p$ is completely determined  by  $S_p$ while the cubic  $C_p$ in $\P^4$ containing $S_p$ is uniquely determined modulo those cubics containing the quadric $Q_p$. Conversely, starting from a $(2,3)$-complete intersection in $\P^4$, the blow up of the sextic surface followed by the contraction of the strict transform of the quadric yields a singular cubic fourfold with a double point, which is the image of the quadric. The type of this singularity depends on the rank of the quadric.

\smallskip

The cubic fourfold $X$ and the complete intersection $S_p$ can both
have singularities, and still the birational transformation here above
holds true. One of the  main results in  \cite{Steg} is that there is a  natural correspondence between singularities on the cubic fourfold and on the sextic surface, including the type of singularities.\par
Suppose that the surface $S_p$ has only isolated singularities. Since the singularities of $S_p$ are simple  the minimal resolution of $S_p$ is a K3 surface. If $S_p$ is singular at a point $y$ then either $y=\pi_p (p')$, with $p' \ne p$ a singular point of $X$,  or  $y$ is singular for the quadric $Q_p$. In this second case the cubic $C_p$ cannot be singular at $y$,
because otherwise $X$ would be singular along the line $\overline {p y}$ while $X$ has only isolated singularities. Therefore the singularities of $X-\{p\}$ are in 1-1 correspondence with the singularities of $S_p$ not contained in $\Sing Q_p$. Let $E_p$ be the exceptional divisor of the the blow-up $\tilde X_p\to X$ at $p$. Then $E_p$ is isomorphic to $Q_p$ and the singularities of $\tilde X_p$ on $E_p$ correspond to the singularities of $S_p$ which are contained in $\Sing Q_p$. If  $X$ has only a single $A_1$ singularity  $p$ then the  surface  $S_p$ is a smooth K3 surface, see \cite[5.1 and 5.2]{Steg}. Note that, due to a result of Varchenko (see \cite[Theorem on the Upper Bound, p. 2781]{Varch}) the maximal number of isolated singularities
which can occur on a cubic fourfold is 15. 
\medskip

 \noindent  Let ${\bf DM}_{\Q}(\mathbf{C})$ be Voevodsky's triangulated category of motives  (with $\Q$-coefficients). There is a fully faithful embedding

\begin{equation}\label{functor} F: \sM_{rat}(\mathbf{C}) \to {\bf DM}_{\Q}(\mathbf{C}).
\end{equation}

\noindent For  every complex  variety $V$ one can define a motive $M(V) \in {\bf DM}_{\Q}(\mathbf{C})$ and every blow-up diagram 

\begin{equation}
\CD E@>{\bar j}>>Y=\Bl_Z \\
@V{\bar \sigma}VV    @V{\sigma}VV \\
Z@>{j}>> X  \endCD \end{equation}

\noindent induces a distinguished triangle in $ {\bf DM}_{\Q}(\mathbf{C})$, see \cite[(4.1.3)]{Voev}

\begin{equation}\label{triangle}\CD  M(E)@>{ (\bar \sigma)_*+\bar j_*}>> M(Z) \oplus M(Y)@>{j_*-  \sigma_*}>>M(X)@>>>M(E)[1] \endCD \end{equation}

\noindent where $E$ is the exceptional divisor. 

\begin{prop}\label{isolated} Let $X\subset \P^5$ be a cubic fourfold with isolated singularities. Assume that  for some singular point $p$ of $X$ the associated  surface $S_p$ belongs to the family $\sG$. Then the motive $M(X)\in{\bf DM}_{\Q}(\mathbf{C})$ is Schur-finite and lies in the triangulated tensor category ${\bf DM}^{Ab}_{\Q}(\mathbf{C})$ generated by the motives of curves.\end{prop}

\begin{proof}  Since $ \tilde X_p $ is the blow-up of $\P^4$ along $S_p$ there is a diagram 
 
 $$\CD E_{S_p}@>>>\tilde X_p=\Bl_{S_p }\\
@VVV    @VVV \\
S_p@>>>\P^4 \endCD $$

\noindent where $E_{S_p}$ is the exceptional  divisor of  the blow-up $q_2: \tilde X_p \to \P^4$. Therefore from (\ref{triangle}) we get a distinguished triangle in ${\bf DM}_{\Q}(\mathbf{C})$

\begin{equation} \label{S_p}\CD  M(E_{S_p})@>>> M(S_p) \oplus M(\tilde X_p)@>>>M(\P^4)@>>>M(E_{S_p})[1]. \endCD\end{equation}

\noindent The exceptional divisor  $E_{S_p}$ is a $\P^1$-bundle over $S_p$. By \ref{symplectic} and \ref{15 nodes}  the motive $h(S_p)$ equals the motive of a smooth K3 surface and is finite-dimensional. Therefore its image in ${\bf DM}_{\Q}(\mathbf{C})$ under the functor (\ref{functor}) is  finite dimensional. By the projective bundle theorem, that is valid also  in $ {\bf DM}_{\Q}(\mathbf{C})$, see \cite[(4.1.11)]{Voev} , we get 

$$M(E_{S_p}) \simeq M(S_p)(1)[2]$$

\noindent Therefore the motives   $M(E_{S_p})$, $M(S_p)$ and $M(\P^4)$ are finite dimensional  in ${\bf DM}_{\Q}(\mathbf{C})$. Finite dimensional objects in the triangulated category  ${\bf DM}_{\Q}(\mathbf{C})$ are also Schur-finite  and Schur-finiteness has the two out of three property in ${\bf DM}_{\Q}(\mathbf{C})$, see \cite[Proposition 5.3]{Maz}. Therefore from (\ref {S_p}) we get that  the motive $M(\tilde X_p)$ is Schur-finite.\par
\noindent The blow-up $q_1 :\tilde X_p \to X$ induces a distinguished triangle in  ${\bf DM}_{\Q}(\mathbf{C})$ as in (\ref{triangle}) 

$$\CD  M(E_p)@>>> M(\{p\}) \oplus M(\tilde X_p)@>>>M(X)@>>>M(E_{p})[1] \endCD.$$

\noindent The exceptional divisor $E_p \subset \tilde X_p$ is isomorphic to the quadric $Q_p \subset \P^4$. Therefore  $M(E_p) $ is Schur-finite in ${\bf DM}_{\Q}(\mathbf{C})$.
The motives $M(E_p)$, $M(\{p\})$ and $M(\tilde X_p)$ are Schur-finite.  By the the two out of three property we get that the motive $M(X)$ is Schur-finite. Moreover, since the motives of $E_p$, $\{p\}$ and $\tilde X_p$ are of Abelian type in $\sM_{rat}(\mathbf{C})$, the motive $M(X)$ belongs to the triangulated tensor subcategory ${\bf DM}^{ab}_{\Q}(\mathbf{C})$ of ${\bf DM}_{\Q}(\mathbf{C})$ generated by the motives of curves, see \cite[Proposition 1.5.6.]{Ay}.

 \end{proof}  

 \subsection{A family of cubic fourfolds with associated 15-nodal, degree six K3 surface.}\label{famnodal} 
 
 The following construction gives an example of a dimension 4 family of singular cubic fourfolds, obtained from K3 surfaces belonging to the family $\sG$ of $15$-nodal K3 surfaces. Therefore, by Proposition \ref{isolated} the motive of all these fourfolds in ${\bf DM}_{\Q}(\mathbf{C})$ is Schur-finite.\par

The mere existence of a family of dimension 4 of $15$-nodal K3 surfaces of degree 6 is assured by \cite[Theorem. 8.3]{GS}, but for the lack of a precise reference of our knowledge, we give here a geometric construction of such surfaces. 
 
\begin{prop}
A sextic, 15-nodal complete intersection surface in $\P^4$ is a double cover of $\P^2$ ramified along a degenerate sextic $F$ given by two conics and two lines. Such surfaces form a 4-dimensional family. The sextic $F$ has two everywhere tangent conics.
\end{prop} 
 
\begin{proof}
Let us consider such a K3 surface $S\subset \P^4$ and project it onto $\P^2$ from a line joining two nodes. Call them $q_1$ and $q_2$. Since both the nodes have multiplicity 2 this exhibits $S$ as a double cover of the plane, ramified along a sextic curve $F$.
 
For the generic desingularized K3 surface $\tilde S$ in our family, we have that $NS(S)\otimes \Q$ is generated by the classes $H, E_1, \dots, E_{15}$, where $H^2=6$ is the natural polarization and the $E_i$ are the exceptional divisors with $E_i^2=-2$. Hence the polarization on $\tilde S$ that gives the map to $\P^2$ is $|H- E_1 - E_2|,$ and the family has dimension 4.
 
Since the projective model in $\P^4$ has $15$ nodes, the ramification sextic cannot be smooth. In fact, $F$ has 13 nodes, that are the images of the nodes $q_3, \dots, q_{15}$, not contained inside the line center of projection. Hence it is reducible, and it splits as two conics and two lines. The images of the two exceptional divisors $E_1$ and $E_2$, exactly as in the well-known case of the double cover of $\P^2$ defined by a Kummer surface, are sent to two conics $C_1$ and $C_2$ in $\P^2$, which are everywhere tangent to F. That is, $C_i \cap F= 2\Delta_i $, where the $\Delta_i$ are degree 6 divisors. The strict transform in $\tilde{S}$ of each $C_i$ is the union of two $(-2)$curves that intersect in 6 points. One of these two curves corresponds to (one of the) node(s) from which we project.
\end{proof}

\begin{rem}
We also observe that a naive parameter count gives $5+5+2+2-8=6$ for the two lines and two conics up to projectivity. The remaining two conditions to descend to the dimension of the family are given by the tangency conditions.
\end{rem}

Now, the construction of the corresponding family of cubic fourfolds is the same as in the one nodal case (see for example \cite[Lemma 5.1]{Kuznetsov_2009}). Take any K3 surface $S\subset \P^4$ in the family $\mathcal{G}$ and consider the map

\begin{equation}\label{cubicmap}
\psi:\P^4 \dashrightarrow \P^5,
\end{equation}

given by the full linear system of cubics through $S$. This induces a birational transformation that consists first in blowing-up $S$, and then contract to one point all the trisecant lines to $S$. The union of these trisecant lines is exactly the strict transform of the  only quadric through $S$, which is then contracted on a (singular) point of $X$, giving the birational transformation already described in Sect. \ref{schurbirat}. By this construction and Proposition  \ref{isolated}, we obtain the following.

\begin{thm}
Let $S\subset \P^4$ be any K3 surface from the family $\mathcal{G}$, and let $X$ be a singular cubic fourfold obtained from $S$ via the map $\psi$ of Eq. \ref{cubicmap}. Then $X$ has Schur finite motive in ${\bf DM}_{\Q}(\mathbf{C})$ and it belongs to the subcategory ${\bf DM}^{Ab}_{\Q}$.  
\end{thm}

\bibliography{bib_tocho}
\bibliographystyle{abbrv}

\end{document}